\documentclass[11pt]{article}
\usepackage[all]{xy}
\usepackage{amsmath}
\usepackage{amsfonts}
\usepackage{amssymb}
\usepackage{amscd}
\usepackage{amsthm}
\usepackage{latexsym}
\usepackage{amsbsy}
\usepackage{xcolor,enumerate}
\usepackage[normalem]{ulem}
\def\0D{\Delta^{(0)}}
\def\1D{\Delta^{(1)}}

\newcommand{\Mc}{\mathcal{M}}
\newcommand{\Zc}{\mathcal{Z}}
\newcommand{\Fc}{\mathcal{F}}

\newcommand{\hmod}{{_H\mathcal{M}}}

\newcommand{\contrakg}{\widehat{\Mc^{kG}}}

\newcommand{\contra}{\widehat{\Mc^H}}

\newtheorem{theorem}{Theorem}[section]
\newtheorem{remark}[theorem]{Remark}
\newtheorem{proposition}[theorem]{Proposition}
\newtheorem{lemma}[theorem]{Lemma}
\newtheorem{corollary}[theorem]{Corollary}

\newtheorem{definition}[theorem]{Definition}

\def\build#1_#2^#3{\mathrel{\mathop{\kern 0pt#1}\limits_{#2}^{#3}}}

\newcommand{\ps}[1]{~\hspace{-3pt}^{^{\left(#1\right)}}}

\newcommand{\vect}{\text{Vec}}
\newcommand{\id}{1}

\newcommand{\la}{\triangleright}
\newcommand{\ra}{\triangleleft}

\newcommand{\contramod}{ctrmd}

\numberwithin{equation}{section}

\def\ot{\otimes}
\def\part{\partial}

\def\ot{\otimes}

\def\Aut{\mathop{\rm Aut}\nolimits}

\def\Hom{\mathop{\rm Hom}\nolimits}

\def\build#1_#2^#3{\mathrel{
\mathop{\kern 0pt#1}\limits_{#2}^{#3}}}

\numberwithin{equation}{section}
\newcommand{\comment}[1]{\relax}

\begin{document}
\title{On the anti-Yetter-Drinfeld module-contramodule correspondence.}
\author {Ilya Shapiro}
\date{
}
\maketitle
\begin{abstract}
We study a functor from anti-Yetter Drinfeld modules to contramodules in the case of a Hopf algebra $H$.  This functor is unpacked from the general machinery of \cite{s}.  Some byproducts of this investigation are the establishment of sufficient conditions for this functor to be an equivalence, verification that the center of the opposite category of $H$-comodules is equivalent to anti-Yetter Drinfeld modules in contrast to \cite{sk} where the question of $H$-modules was addressed, and the observation of two types of periodicities of the generalized Yetter-Drinfeld modules introduced in  \cite{hks}.  Finally, we give an example of a symmetric $2$-contratrace on $H$-comodules that does not arise from an anti-Yetter Drinfeld module.

\end{abstract}

\medskip


\section{Introduction.}
This paper grew out of the author's attempts to better understand contramodules at least in some simple examples.  The simplest case being the Hopf algebra $kG$ where $G$ is a discrete infinite group.  Contramodules over a coalgebra were introduced by Eilenberg and Moore in 1965 and can be viewed either as algebraic structures allowing infinite combinations or a better behaved notion than that of modules over the dual algebra (see Remark \ref{contradiff}).  They do not strictly speaking generalize comodules, but do have a non-trivial intersection with them.  In our investigations we found \cite{pos} to be very helpful, in fact the phenomenon of this underived comodule-contramodule correspondence without the anti-Yetter-Drinfeld enhancement is investigated there as well.

The introduction of anti-Yetter-Drinfeld contramodule coefficients to the Hopf-cyclic cohomology theory in \cite{contra} that followed the definition of anti-Yetter-Drinfeld module coefficients in \cite{hkrs2} can in retrospect be conceptually understood as  being completely natural since they are seen to be exactly corresponding to the representable symmetric $2$-contratraces, see \cite{s} and \cite{sk}.  The latter form a well behaved class of Hopf-cyclic coefficients explored in \cite{hks} and \cite{s}, that lead directly to Hopf-cyclic type cohomology theories.

Roughly speaking, the category of stable anti-Yetter-Drinfeld modules consists of $H$-modules and comodules such that the two structures are compatible in a way that ensures that they form the center of a certain bimodule category \cite{hks}.  A similar statement with contramodule structure replacing the comodule one can be made about   anti-Yetter-Drinfeld contramodules.  In general understanding objects in these categories is not a simple task, however in the case of $H=kG$ the former category is known to consist of $G$-graded $G$-equivariant vector spaces, i.e., $\oplus_{g\in G}M_g$ with $x: M_g\to M_{xgx^{-1}}$.  Stability, a condition that ensures cyclicity, translates to $x=Id_{M_x}$.  We could find no similarly simple description of the anti-Yetter-Drinfeld contramodule category in the literature.  It turns out, Corollary \ref{kgeq}, that this category is also equivalent to $G$-graded $G$-equivariant vector spaces but the objects are now $\prod_{g\in G}M_g$.  The Theorem \ref{semisimple} is a more general case of this correspondence.

The above anti-Yetter-Drinfeld module-contramodule correspondence was the motivation for the rest of the results in this paper.  Namely, the Proposition \ref{semisimplecase} shows that the equivalence arises from a functor $M\mapsto\widehat{M}$ from comodules to contramodules.  This functor can be found in \cite{pos} but arose independently from the considerations of \cite{s} which furthermore demonstrate that it works on the anti-Yetter-Drinfeld versions as well.  More precisely, for $M$ a stable anti-Yetter-Drinfeld module we consider $\Fc(-)=\Hom(-,M)^H$ which is a symmetric $2$-contratrace on $H$-comodules, i.e., a contravariant functor from $\Mc^H$ to $\vect$ subject to a trace-like symmetry. Its pullback to the category $\hmod$ of $H$-modules is $\Hom_H(-,\widehat{M})$.  The pullback construction reduces in this case to the observation that $H\in\Mc^H$ is an algebra and the category of $H$-bimodules in $\Mc^H$ is equivalent to $\hmod$.  The pullback $\Hom_H(-,\widehat{M})$ is obtained as $\Fc_H$, i.e., the equalizer of the action maps $\Fc(V)\to\Fc(H\ot V)$ and $\Fc(V)\to\Fc(V\ot H)$ with the targets identified via the symmetry of $\Fc$.  Though this can be used as the definition of $\widehat{M}$, we give an explicit construction of both the contramodule structure, essentially agreeing with \cite{pos}, and the $H$-action on $\widehat{M}=\Hom(H,M)^H$.

It turns out that, not surprisingly, $M\mapsto\widehat{M}$ is not always an equivalence, but it does have a left adjoint, that we found in \cite{pos} and upgraded to the anti-Yetter-Drinfeld setting here.  The key object when studying the question of equivalence is the ideal of left integrals for $H$ as introduced in \cite{integral}.  This object seems to be the first example of a generalized Yetter-Drinfeld module of charge other than $1$ or $-1$, its charge is $2$. These were introduced in \cite{hks} without any hope that anything other than $\pm 1$ would be useful.  In fact, the conditions for the comodule-contramodule correspondence are closely related to the presence of a $2$-periodicity of the charges, see  Remark \ref{periodicity 1}.  Furthermore, in studying the question of stability of anti-Yetter-Drinfeld modules/contramodules and the generalization of this concept to more general charges (in a way that was necessarily different from \cite{hks}) we observed a second kind of periodicity within a generalized Yetter-Drinfeld category of a fixed charge.  The Remark \ref{second periodicity} describes an action of $\mathbb{Z}/i\mathbb{Z}$ on Yetter-Drinfeld modules of charge $i-1$ and Yetter-Drinfeld contramodules of charge $i+1$.  This action is compatible with the generalized $M\mapsto \widehat{M}$ that sends Yetter-Drinfeld modules of charge $i-1$ to Yetter-Drinfeld contramodules of charge $i+1$; the case of $i=0$ is the usual anti-Yetter-Drinfeld situation.

Identifying categories of interest with centers of bimodule categories such as was done in \cite{hks} and \cite{sk} is carried through this paper as well.  We point to the summary Theorem \ref{summarythm} that is of the \cite{hks} flavor, and the  Corollary \ref{saydmods} of the \cite{sk} flavor as examples.  One of the natural questions that arose after \cite{hks} was if symmetric $2$-contratraces give a more general class of coefficients for Hopf-cyclic type theories even in the case of Hopf algebras.  It turns out \cite{sk} that for the case of $\hmod$ the representable symmetric $2$-contratraces are equivalent to the  anti-Yetter-Drinfeld contramodules, and similarly (Corollary \ref{reptrace1}) for the case of $\Mc^H$ the representable symmetric $2$-contratraces are equivalent to the anti-Yetter-Drinfeld modules.  Thus one needs only find a non-representable example of a contratrace in order to have a new coefficient in the $H$-comodule case.  This is explained in Section \ref{nonrep}.

The paper is arranged as follows: Section \ref{hcomodules} is devoted to the establishment of the fact that $\Mc^H$ is biclosed, and thus it makes sense to consider the opposite bimodule category $\Mc^{H op}$ with the adjoint action.  The center is shown to consist of  anti-Yetter-Drinfeld modules; this identifies symmetric representable $2$-contratraces with stable anti-Yetter-Drinfeld modules.  Finally we introduce the functor $M\mapsto\widehat{M}$.  In Section \ref{adjunctionssection} we demonstrate that the adjoint pair of functors between comodules and contramodules: $N\mapsto N'$ and $M\mapsto\widehat{M}$ induce the same between the anti-Yetter-Drinfeld versions and also the stable anti-Yetter-Drinfeld versions.  Section \ref{correspondence} deals with the question of equivalence established by $M\mapsto\widehat{M}$ and ends with an example of a new coefficient for the case of $H=kG$.  In Section \ref{period} we extend $M\mapsto\widehat{M}$ to the generalized Yetter-Drinfeld modules and discuss two types of periodicities and the compatibility of  $M\mapsto\widehat{M}$ with them.

Some things to keep in mind: for a coalgebra $C$ we use the following version of Sweedler notation $\Delta(c)=c^1\ot c^2$.  For a right comodule $N$ over $C$ we use $\rho(n)=n_0\ot n_1$.  All Hopf algebras have invertible antipodes $S$ and are over a field $k$ of characteristic $0$.  We denote by $\Hom(-,-)^H$ and $\Hom_H(-,-)$ the morphisms in $\Mc^H$ and $\hmod$ respectively, while $\Hom(-,-)$ stands for $k$-linear maps.

\bigskip

\textbf{Acknowledgments}: The author wishes to thank Masoud Khalkhali and Piotr Hajac for stimulating questions and discussions.   This research was supported in part by the NSERC Discovery Grant number 406709.

\bigskip

\section{The category of $H$-comodules.}\label{hcomodules}

This section is primarily dedicated to  the establishment of the fact that the monoidal category $\Mc^H$ of $H$-comodules is biclosed, and the analysis of the center of the bimodule category over $\Mc^H$ resulting from considering $\Mc^{H op}$.  This establishes an analogue of a result in \cite{sk} describing the center as the category of $aYD$-modules for $H$.

\subsection{Internal Homs in the category of $H$-comodules.}\label{comodules}
Motivated by the existence of internal Homs in $\hmod$, and thus the possibility of describing representable contratraces on $\hmod$ as central elements in the opposite category, we will now address the same question in $\Mc^H$, the monoidal category of $H$-comodules.

Since in the finite dimensional $H$ case, we have that $\Mc^H\simeq{_{H^*}\Mc}$ so we have a suggestive way of obtaining the required formulas. We note that some modifications do need to be made to account for possible infinite dimensionality of $H$.

For $W,V\in\Mc^H$ consider $\rho:\Hom(W,V)\to\Hom(W,V\ot H)$ given by  \begin{equation}\label{coactionleft}\rho f(w)=f(w_0)_0\ot f(w_0)_1 S(w_1).\end{equation}

\begin{definition}
Define $\Hom^l(W,V)$ as the subspace of $\Hom(W,V)$ that consists of $f$ such that $\rho f\in\Hom(W,V)\ot H$.
\end{definition}

We can define two maps $$``Id\ot\Delta"=(Id\ot\Delta)\circ -$$ and $$``\rho\ot Id"=(Id_V\ot m\ot Id_H)\circ(Id_{V\ot H}\ot\sigma_{H,H})\circ(\rho_V\ot Id_{H\ot H})\circ(f\ot S)\circ\rho_W$$ from $\Hom(W,V\ot H)$ to $\Hom(W,V\ot H\ot H)$.  The latter can be written down more manageably as follows: let  $f(w)=w\ps{1}\ot w\ps{2}$ then $$``\rho\ot Id"(f)(w)=((w_0)\ps{1})_0\ot((w_0)\ps{1})_1 S(w_1)\ot(w_0)\ps{2}.$$  A direct computation shows that \begin{equation}\label{ass}``Id\ot\Delta"\circ\rho=``\rho\ot Id"\circ\rho.\end{equation}  Note that when restricted to $\Hom(W,V)\ot H$ the maps $``Id\ot\Delta"$ and $``\rho\ot Id"$ are actually $Id\ot\Delta$ and $\rho\ot Id$ respectively.  The formula \eqref{ass} has two important and immediate consequences: $\rho:\Hom^l(W,V)\to\Hom^l(W,V)\ot H$, whereas before we only knew that it lands in $\Hom(W,V)\ot H$, and $\rho$ is a coaction.

It is not hard to see that $\Hom^l(W,V)$ is contravariant in $W$ and covariant in $V$.  More precisely, let $\phi\in\Hom(W',W)^H$ and $\theta\in\Hom(V,V')^H$, then the following diagram commutes: $$\xymatrix{\Hom(W,V)\ar[rr]^{\theta\circ -\circ\phi}\ar[d]^{\rho}& & \Hom(W',V')\ar[d]^{\rho}\\ \Hom(W,V\ot H)\ar[rr]^{(\theta\ot Id)\circ -\circ\phi} & & \Hom(W',V'\ot H) \\
}$$ and so $\theta\circ -\circ\phi:\Hom^l(W,V)\to\Hom^l(W',V')$ and it is a map of $H$-comodules.

\begin{lemma}\label{adjunction}
We have natural identifications $$\Hom(T\ot W,V)^H\simeq\Hom(T,\Hom^l(W,V))^H,$$ i.e., $\Hom^l(-,-)$ is the left internal Hom in $\Mc^H$.
\end{lemma}
\begin{proof}
Note that $f\in\Hom(T\ot W,V)^H$ if and only if \begin{equation}\label{invar1}f(t\ot w)_0\ot f(t\ot w)_1=f(t_0\ot w_0)\ot t_1 w_1.\end{equation} On the other hand $\phi\in\Hom(T,\Hom^l(W,V))^H$ if and only if \begin{equation}\label{invar2}\rho\phi_t=\phi_{t_0}\ot t_1 \end{equation} where $\phi_t\in\Hom^l(W,V)$.

Let $f\in\Hom(T\ot W,V)^H$ then if we define $f_t(w)=f(t\ot w)$ we have \begin{align*}\rho f_t(w)&=f(t\ot w_0)_0\ot f(t\ot w_0)_1 S(w_1)\\
&=f(t_0\ot w_{0,0})\ot t_1 w_{0,1}S(w_1)\\
&=f(t_0\ot w_0)\ot t_1 w_1 S(w_2)\\
&=f(t_0\ot w)\ot t_1\\
&=f_{t_0}\ot t_1.
\end{align*}  So that $t\mapsto f_t\in\Hom(T,\Hom^l(W,V))^H$.  Conversely, if $\phi\in\Hom(T,\Hom^l(W,V))^H$ then define $\phi(t\ot w)=\phi_t(w)$ then $\phi(t_0\ot w)\ot t_1= \phi(t\ot w_0)_0\ot \phi(t\ot w_0)_1 S(w_1)$ so that \begin{align*}\phi(t_0\ot w_0)\ot t_1 w_1
&=\phi(t\ot w_{0,0})_0\ot\phi(t\ot w_{0,0})_1 S(w_{0,1})w_1\\
&=\phi(t\ot w_{0})_0\ot\phi(t\ot w_{0})_1 S(w_{1})w_2\\
&=\phi(t\ot w)_0\ot\phi(t\ot w)_1
\end{align*} and thus $t\ot w\mapsto\phi_t(w)\in\Hom(T\ot W,V)^H$.

So the usual bijection $f(t\ot w)=f_t(w)$ establishes a natural identification between  $\Hom(T\ot W,V)^H$ and $\Hom(T,\Hom^l(W,V))^H$ as required.
\end{proof}

\begin{remark}
From now on we will denote the coaction $\rho$ of \eqref{coactionleft} by $\rho^l$ since $$\rho^l:\Hom^l(W,V)\to\Hom^l(W,V)\ot H. $$
\end{remark}

\begin{remark}
Note that we have a natural fully faithful embedding of $\Mc^H$ into ${_{H^*}\Mc}$.  The right adjoint to it can be used to define $\Hom^l(W,V)$.  Namely, the formula \eqref{coactionleft} defines a left $H^*$-module structure on $\Hom(W,V)$ via $\chi f=(Id_V\ot\chi)(\rho f)$.  Then it is easy to see that $$\Hom(W,V)^{rat}=\Hom^l(W,V)$$
where $(-)^{rat}$is the right adjoint that features prominently in \cite{integral}.\end{remark}

Repeating the above considerations nearly verbatim, we define the right internal Hom for $\Mc^H$ as follows.  Begin by defining
\begin{equation}\label{coactionright}\rho^r f(w)=f(w_0)_0\ot S^{-1}(w_1)f(w_0)_1 .\end{equation}

\begin{definition}
Define $\Hom^r(W,V)$ as the subspace of $\Hom(W,V)$ that consists of $f$ such that $\rho^r f\in\Hom(W,V)\ot H$.
\end{definition}

We again obtain that $\rho^r:\Hom^r(W,V)\to\Hom^r(W,V)\ot H$ and is a coaction.  Furthermore, we have natural adjunctions: $$\Hom(W\ot T,V)^H\simeq\Hom(T,\Hom^r(W,V))^H.$$

As usual we now have the opposite category $\Mc^{H op}$ with $$V\ra W= \Hom^r(W, V), \quad ~~ \text{and} \quad ~~ W\la V=\Hom^l(W, V)$$ and we may examine its center $\Zc_{\Mc^H}(\Mc^{H op})$.

\begin{remark}
We observe that if $V\in\Mc^H$ is finite dimensional then $$\Hom^l(V,W)=W\ot V^*$$ and $$\Hom^r(V,W)={^*V}\ot W$$ where $V^*=\Hom^l(V,k)$ and ${^*V}=\Hom^r(V,k)$ and both $V^*$ and ${^*V}$ are $\Hom_k(V,k)$ as vector spaces.  So that $\Mc^H_{fd}$ is rigid with $${^{**}V}=V^{S^{-2}}\quad\text{and}\quad V^{**}=V^{S^2},$$ where $V^{S^{2i}}$ denotes the $H$-comodule with the coaction modified by $S^{2i}$.
\end{remark}

\subsection{The center of the opposite bimodule category.}

We recall  from \cite{hkrs2}  that   a left-right anti-Yetter-Drinfeld module $M$ over a Hopf algebra $H$   is  a left $H$-module and a right $H$-comodule satisfying
 \begin{equation}\label{AYD1}
   (hm)_{0}\ot (hm)_{1}=  h^{2}m_{0}\ot h^{3}m_{1}S(h^{1}).
 \end{equation}  It is stable if $m_1 m_0=m$ for all $m\in M$.

Recall, for example from \cite{s}, the notions of the opposite bimodule category and  of the center of a bimodule category.

\begin{proposition}\label{aydprop2}
The category of $aYD$-modules for $H$ is equivalent to $\Zc_{\Mc^H}(\Mc^{H op})$.

\end{proposition}

\begin{proof}
Let $M$ be an $aYD$-module and define the central structure $$\tau:\Hom(W,M)\to\Hom(W,M)$$ $$\tau f(w)=w_1 f(w_0).$$  Note that $\tau$ is invertible with $\tau^{-1}f(w)=S^{-1}(w_1)f(w_0)$.  Define a map $``\tau\ot Id":\Hom(W,M\ot H)\to\Hom(W,M\ot H)$ by $$``\tau\ot Id"=(a\ot Id)\circ\sigma_{M,H,H}\circ(f\ot Id)\circ\rho_W$$ so that if $f(w)=w\ps{1}\ot w\ps{2}$ then $$``\tau\ot Id"(f)(w)=w_1(w_0)\ps{1}\ot(w_0)\ps{2}\!\!\!\!\!\!.$$  A direct computation (using the $aYD$ condition \eqref{AYD1}) demonstrates that the following diagram commutes:\begin{equation}\label{coco}\xymatrix{\Hom(W,M)\ar[r]^\tau\ar[d]^{\rho^l} & \Hom(W,M)\ar[d]^{\rho^r}\\
\Hom(W,M\ot H)\ar[r]^{``\tau\ot Id"} & \Hom(W,M\ot H)\\
}\end{equation}
and since $``\tau\ot Id"$ restricted to $\Hom(W,M)\ot H$ is actually $\tau\ot Id$, so in fact $\tau:\Hom^l(W,M)\to\Hom^r(W,M)$ and it is an isomorphism in $\Mc^H$.

Observe that $\tau$ is natural in $W$, i.e., if $\phi:W'\to W$ is a morphism in $\Mc^H$ then $\phi(w_0)\ot w_1=\phi(w)_0\ot\phi(w)_1$ so that $w_1 f(\phi(w_0))=\phi(w)_1 f(\phi(w)_0)$ and $\tau(f\circ\phi)=\tau f\circ\phi$.

If $\theta: M\to M'$ is a map of $aYD$-modules, then $$\theta\circ\tau f(w)=\theta(w_1 f(w_0))=w_1\theta f(w_0)=\tau(\theta\circ f)(w)$$ so that a map of $aYD$-modules induces a map of central elements.

To check the commutativity of \begin{equation}\label{assoc}\xymatrix{W\la(V\la M)\ar[r]^{Id\la\tau}& W\la(M\ra V)\ar[r]^{\tau\ra Id} &(M\ra W)\ra V\ar[d]\\
(W\ot V)\la M\ar[u]\ar[rr]^{\tau} & &M\ra(W\ot V) \\
}\end{equation} is to check that going along the bottom and obtaining $w\ot v\mapsto (w_1v_1)f(w_0\ot v_0)$ is the same as the long way around which gives $w\ot v\mapsto w_1(v_1 f_{w_0}(v_0))=w_1(v_1 f(w_0\ot v_0))$; and they are the same by the usual $H$-action axiom.  Similarly, the unitality of the action implies that $k\la M\to M\ra k$ is the identity since $\tau:m\mapsto 1m$.

What has been shown so far is that if $M$ is an $aYD$-module, then $(M,\tau)\in\Zc_{\Mc^H}(\Mc^{H op})$ and any $\theta:M\to M'$ a morphism of $aYD$-modules induces a morphism between the corresponding central elements.

Conversely, let $M\in\Zc_{\Mc^H}(\Mc^{H op})$ so that we have natural isomorphisms $\tau:\Hom^l(W,M)\to\Hom^r(W,M)$.  Note that  $$\Hom(W,M)=\varprojlim_\alpha\Hom^l(W_\alpha,M)$$ where $W_\alpha\subset W$ is a finite dimensional sub-comodule since any $w\in W$ is contained in such an $W_\alpha$ and so $$\Hom(W,M)=\Hom(\varinjlim_\alpha W_\alpha, M)=\varprojlim_\alpha\Hom(W_\alpha,M)=\varprojlim_\alpha\Hom^l(W_\alpha,M).$$  So we have a $\tau:\Hom(W,M)\to\Hom(W,M)$ that satisfies a version of all the properties that make the original $\tau$ so useful.  Denote by $r$ the composition $$M\to\Hom(H,M)\to\Hom(H,M)$$  so that $m\mapsto \tau(h\mapsto\epsilon(h)m)$. Define \begin{equation}\label{action}hm=r_m(h).\end{equation}  Note that we needed to use $\Hom(H,M)$ instead of $\Hom^l(H,M)$ since $h\mapsto \epsilon(h)m$ is not in $\Hom^l(H,M)$.

By the unitality of $\tau$ we have $ev_1\circ\tau=ev_1$ so that $$1m=r_m(1)=ev_1\tau(h\mapsto\epsilon(h)m)=ev_1(h\mapsto\epsilon(h)m)=\epsilon(1)m=m.$$

Furthermore by the ``associativity" of $\tau$, i.e., the diagram \eqref{assoc} and its naturality, we have \begin{align*}(xy)m&=r_m(xy)\\
&=\tau(h\mapsto\epsilon(h)m)(xy)\\
&=\tau(h\ot h'\mapsto\epsilon(hh')m)(x\ot y)\\
&=\tau(h\mapsto\tau(h'\mapsto\epsilon(hh')m)(y))(x)\\
&=\tau(h\mapsto\epsilon(h)\tau(h'\mapsto\epsilon(h')m)(y))(x)\\
&=\tau(h\mapsto\epsilon(h)r_m(y))(x)\\
&=r_{r_m(y)}(x)\\
&=x(ym).
\end{align*}

Let $\theta: M\to M'$ be a map in the center, then we have $$\xymatrix{M\ar[r]^{\!\!\!\!\!\!\!\!\!\!\!\!\!\!\!\!-\circ\epsilon}\ar[d]^\theta & \Hom(H,M)\ar[r]^{\tau}\ar[d]^{\theta\circ -} & \Hom(H,M)\ar[d]^{\theta\circ -}\\
M'\ar[r]^{\!\!\!\!\!\!\!\!\!\!\!\!\!\!\!\!-\circ\epsilon} & \Hom(H,M')\ar[r]^{\tau} & \Hom(H,M')\\
}$$ where the left square commutes trivially and the right one commutes by definition, so that $$\theta(hm)=\theta(r_m(h))=r_{\theta(m)}(h)=h\theta(m).$$

Before proving that the $H$-action defined above satisfies the $aYD$-module condition \eqref{AYD1} we will show that the definition of the action from $\tau$ and vice versa are mutually inverse.  Let an $H$-action be given, then we set $\tau f(w)=w_1 f(w_0)$ so that the action becomes $$r_m(h)=\tau(x\mapsto\epsilon(x)m)(h)=h^2\epsilon(h^1)m=hm,$$ i.e., the original action. On the other hand if $\tau:\Hom(W,M)\to\Hom(W,M)$ is given and we defined the action by $hm=r_m(h)=\tau(x\mapsto\epsilon(x)m)(h)$, then we obtain the following.  Let $f\in\Hom(W,M)$, consider $f\ot\epsilon\in\Hom(\underline{W}\ot H, M)$ where $\underline{W}$ is a trivial $H$-comodule.  Note that the co-action map $\rho_W$ is a morphism in $\Mc^H$ from $W$ to $\underline{W}\ot H$.  So $\tau(f\ot\epsilon\circ\rho_W)=\tau(f\ot\epsilon)\circ\rho_W$ and the former is $\tau f$ while the latter is \begin{align*}w&\mapsto \tau(w\ot h\mapsto f(w)\epsilon(h))(w_0\ot w_1)\\
&=\tau(w\mapsto \tau(h\mapsto f(w)\epsilon(h))(w_1))(w_0)\\
\intertext{and since the coaction of $H$ on $\underline{W}$ is trivial so}
&=\tau(h\mapsto\epsilon(h)f(w_0))(w_1)\\
&=w_1f(w_0).
\end{align*}  So that no matter if we start with a $\tau$ or an $H$-action, we always have \begin{equation}\label{tau}\tau f(w)=w_1f(w_0).\end{equation}

Now recall the diagram \eqref{coco}, and note that it now commutes essentially by definition.  Let $W=H$ and keep in mind the formula \eqref{tau}.  We now get that for any $f\in\Hom(H,M)$ we have $$h^3f(h^1)_0\ot f(h^1)_1 S(h^2)=(h^2f(h^1))_0\ot S^{-1}(h^3)(h^2f(h^1))_1$$ and let us apply it to $f(h)=\epsilon(h)m$ to obtain $$h^2m_0\ot m_1S(h^1)=(h^1m)_0\ot S^{-1}(h^2)(h^1m)_1$$ so that \begin{align*}
h^2m_0\ot h^3m_1S(h^1)&=(h^1m)_0\ot h^3S^{-1}(h^2)(h^1m)_1\\
&=(h^1m)_0\ot \epsilon(h^2)(h^1m)_1\\
&=(hm)_0\ot(hm)_1
\end{align*} and $M$ satisfies the $aYD$-module condition \eqref{AYD1}.

\end{proof}

Recall that we denote by $\Zc'_{\Mc^H}(\Mc^{H op})$ the full subcategory that consists of objects such that the identity map $Id\in\Hom(M,M)^H$ is mapped to itself via \begin{equation}\label{stabilitycond}\Hom(M,M)^H\simeq\Hom(\id,M\la M)^H\simeq\Hom(\id,M\ra M)^H\simeq \Hom(M,M)^H.\end{equation} We have a straightforward corollary:

\begin{corollary}\label{saydmods}
The category of $saYD$-modules for $H$ is equivalent to $\Zc'_{\Mc^H}(\Mc^{H op})$.
\end{corollary}

\begin{proof}
Recall that an $aYD$-module $M$ is stable if $m_1 m_0=m$ for all $m\in M$.  On the other hand considering $\tau:\Hom(M,M)\to \Hom(M,M)$ we see that according to \eqref{tau} we have $\tau Id(m)=m_1m_0$ and so $\tau Id=Id$ if and only if $M$ is stable.
\end{proof}

Thus we have established the following:
\begin{corollary}\label{reptrace1}
The category of $saYD$-modules for $H$ is equivalent to the category of representable symmetric $2$-contratraces on $\Mc^H$ via $$M\longleftrightarrow\Hom(-,M)^H.$$
\end{corollary}

Contrast that with the $\hmod$ case considered in \cite{sk} where the category of representable symmetric $2$-contratraces is equivalent to the more unusual $saYD$-\emph{contra}modules.

\subsection{A functor from $(s)aYD$-modules to $(s)aYD$-contramodules}\label{functorsec}
This section is motivated by the adjunction on cyclic cohomology of \cite{s} that we explain below.
Given an $saYD$-module $M$, i.e., a representable symmetric $2$-contratrace $\Hom(-,M)^H$, as a special case of the theory developed in \cite{s}, we obtain an $H$-module $\widehat{M}$ such that $\Hom_H(-,\widehat{M})$ is a representable symmetric $2$-contratrace.

We will need to recall from \cite{contra} that a right $C$-contramodule $N$, where $C$ is a counital coassociative coalgebra, is equipped with the contraaction $$\alpha:\Hom(C,N)\to N$$ satisfying \begin{equation}\label{assoccontra}\alpha(x\mapsto\alpha(h\mapsto f(x\ot h)))=\alpha(h\mapsto f(h^1\ot h^2))\end{equation} for any $f\in\Hom(C\ot C,N)$ and \begin{equation}\label{unitcontra}\alpha(h\mapsto\epsilon(h)n)=n\end{equation} for any $n\in N$. Furthermore, a left-right $aYD$-contramodule $N$ is a left $H$-module and a right $H$-contramodule such that for all $h\in H$ and any linear map $f\in \Hom(H, N)$ we have \begin{equation}\label{aydeq}h\alpha(f)=\alpha(h^2 f(S(h^3)-h^1)).\end{equation} It is called stable, i.e., an $saYD$-contramodule, if for all $n\in N$ we have $\alpha(r_n)=n$ where $r_n(h)=hn$.

We will also recall the definitions from \cite{s}: if $M$ is an $aYD$-module then $$\widehat{M}=\Hom(H,M)^H$$ has a left $H$-action via \begin{equation}\label{weird action}h\cdot\phi(-)=h^2\phi(S(h^3)-h^1)\end{equation} and furthermore $\widehat{M}$ has a contraaction $\alpha:\Hom(H,\widehat{M})\to\widehat{M}$ defined as follows.  Let $\theta\in\Hom(H,\widehat{M})$ be viewed as $h\mapsto\theta_h(-)$ then \begin{equation}\label{contraaction}\alpha\theta(h)=\theta_{h^1}(h^2).\end{equation}  It is not hard to check all these statements directly (note that the $aYD$-module condition for $M$ is only used to ensure that the action \eqref{weird action} preserves the $H$-comodule morphisms inside $\Hom(H,M)$), and most importantly we can also check that $\alpha$ is compatible with the action in the $aYD$-contramodule sense, i.e., the identity \eqref{aydeq} holds.

The constructions above describe a functor \begin{equation}\label{thefunctor}M\mapsto\widehat{M}\end{equation} from $(s)aYD$-modules to $(s)aYD$-contramodules. Furthermore, the functor \eqref{thefunctor} is a special case of the pullback of contratraces \cite{s} and so  we have the following:

\begin{proposition}\label{dualityprop1}
Given an $H$-module algebra $A$ and a $saYD$-module $M$, we have an isomorphism of cyclic cohomologies:
$$\widehat{HC}^n_H(A,\widehat{M})\simeq HC^{n,H}(A\rtimes H,M)$$ where the theories considered are of the derived type.
\end{proposition}

We denote by $\widehat{HC}^n_H(A,\widehat{M})$ the cyclic cohomology obtained from an algebra $A$ and a $saYD$-contramodule $\widehat{M}$ via the associated representable symmetric $2$-contratrace  $\Hom_H(-,\widehat{M})$ on  $\hmod$,  while $HC^{n,H}(A\rtimes H,M)$ denotes the Hopf-cyclic cohomology of an $H$-comodule algebra $A\rtimes H$ with coefficients in a $saYD$-module $M$ obtained from the representable symmetric $2$-contratrace $\Hom(-,M)^H$ on $\Mc^H$.

\begin{remark}
In light of the Corollary \ref{reptrace1} that shows the equivalence between $saYD$-modules and representable symmetric $2$-contratraces on $\Mc^H$ and \cite{sk} where a similar result is demonstrated for $saYD$-contramodules and $\hmod$, the Proposition \ref{dualityprop1} is a concrete realization of the pullback of \emph{representable} contratraces of \cite{s}.
\end{remark}

\section{An adjoint pair of functors.}\label{adjunctionssection}
We will now analyze the functor $M\mapsto\widehat{M}$ with a view towards establishing some sufficient conditions for it being an equivalence. Consider the category $\Mc^H$ of right $H$-comodules and we are interested in comparing it to the category $\contra$ of right $H$-contramodules.  It turns out that the functor $M\mapsto \widehat{M}$, that appeared in \cite{s} motivated by the pullback of contratraces has already appeared in the literature on comodule-contramodule correspondences \cite{pos}, but considered without the extra $H$-module structure that we need.  We will abuse notation somewhat and not usually distinguish between $\widehat{(-)}:\Mc^H\to\contra$ of \cite{pos} and the upgraded version of \cite{s} mentioned above \eqref{thefunctor}.  When we do want to emphasize the difference, the latter will be denoted by $\widehat{(-)}_H$.

Furthermore,  $\widehat{(-)}$ has a left adjoint \cite{pos} \begin{equation}\label{leftadjoint} N\mapsto N' \end{equation} where $N'=H\odot_H N$ is the cokernel of the difference between the maps $Id\ot\alpha$ and $h\ot f\mapsto h^2\ot f(h^1)$ between $H\ot_k\Hom(H,N)$ and $H\ot_k N$: $$H\ot_k\Hom(H,N)\to H\ot_k N\to H\odot_H N\to 0.$$  The comodule structure on $N'$ is given by \begin{equation}\label{coaction}(h\ot n)_0\ot(h\ot n)_1=(h^1\ot n)\ot h^2.\end{equation} When $H$ is finite dimensional then $N'=H\ot_{H^*}N$ so that the notation $\odot_H$ is a bit misleading.

The adjunctions are $$H\odot_H\Hom(H,M)^H\to M$$ $$h\ot f\mapsto f(h)$$ and $$N\to\Hom(H, H\odot_H N)^H$$ $$n\mapsto \{h\mapsto h\ot n\}.$$

\begin{remark}
Just as the functor $M\mapsto\widehat{M}$ was upgraded from the functor between comodules and contramodules to a functor between $aYD$-modules and $aYD$-contramodules by converting an $H$-action on $M$ to an $H$-action on $\widehat{M}$, we can do the same to its left adjoint directly.  Namely, define an $H$-action on $H\ot_k N$ via \begin{equation}\label{notsoweird action}x\cdot(h\ot n)=x^3hS(x^1)\ot x^2n\end{equation} then one can check that if $N$ is an $aYD$-contramodule, then the action is well defined on the cokernel $H\odot_H N$ and gives $N'$ the $aYD$-module structure.
\end{remark}

We will now conceptually investigate if the adjoint pair of the functors above is compatible with the extra structure that we require.  More precisely, $\Mc^H$ is a tensor category in the usual way with $$\rho(m\ot n)=m_0\ot n_0\ot m_1n_1$$ for $m\ot n\in M\ot N$ with $M,N\in\Mc^H$.  Thus $\Mc^H$ is a bimodule category over itself.

On the other hand if $N\in\contra$ and $T\in\Mc^H_{fd}$, i.e., $T$ is a finite dimensional $H$-comodule, then we can define a natural contramodule structure on both $N\ot T$ and $T\ot N$.  Namely, due to the finite dimensionality of $T$, we represent elements of $\Hom(H,N\ot T)$ by $f\ot t$ with $f\in\Hom(H,N)$ and $t\in T$, then \begin{equation}\label{contratensorright}\alpha_{N\ot T}(f\ot t)=\alpha_N(f(-t_1))\ot t_0\end{equation} and similarly \begin{equation}\label{contratensorleft}\alpha_{T\ot N}(t\ot f)=t_0\ot\alpha_N(f(t_1-)) \end{equation} which makes $\contra$ into a bimodule category over $\Mc^H_{fd}$.

The following is the key technical result of this section.  It describes the exact nature of the compatibility of $M\mapsto\widehat{M}$ with the $\Mc^H_{fd}$-bimodule structure on both sides.

\begin{proposition}\label{bimoduleprop}
Let $W\in\Mc^H$ and $T,L\in\Mc^H_{fd}$ then we have: $$\Hom(H,T\ot W\ot L)^H\simeq T^{S^2}\ot\Hom(H,W)^H\ot L^{S^{-2}}$$ $$t\ot f\ot l\mapsto t_0\ot f(S^2(t_1)-S^{-2}(l_1))\ot l_0$$  a natural isomorphism in $\contra$.
\end{proposition}

\begin{proof}
Recall that $\Hom^L(H,W)$ has a left $H^*$-action and a right $H$-contraaction and they commute.  Namely, $$(\chi\cdot f)(h)=f(h^1)_0\chi(f(h^1)_1 S(h^2))$$ and $$\alpha(h\mapsto\theta_h(-))=\{h\mapsto\theta_{h^1}(h^2)\}.$$

One quickly checks that the map $$\Hom^L(H, T\ot W)\to\overline{T}\ot\Hom^L(H,W)$$ \begin{equation}\label{iso1}t\ot f\mapsto t\ot f\end{equation} is an isomorphism of both $H^*$-modules and $H$-contramodules, where $\overline{T}$ has the usual $H^*$ structure, but is considered trivial for the purposes of defining the $H$-contraaction on the right hand side.

On the other hand $$\Hom^L(H,W)\ot\overline{T}\to\underline{T}\ot\Hom^L(H,W)$$ \begin{equation}\label{iso2}f\ot t\mapsto t_0\ot f(t_1-)\end{equation} is also an isomorphism of both structures where $\underline{T}$ has trivial $H^*$ structure but non-trivially modifies the contraaction on the right hand side.

So as $H$-contramodules we have: \begin{align*} \Hom(H,T\ot W)^H&\simeq\Hom_{H^*}(k,\Hom^L(H,T\ot W))\\
&\simeq\Hom_{H^*}(k,\overline{T}\ot\Hom^L(H,W))\\
&\simeq\Hom_{H^*}(k,\Hom^L(H,W)\ot\overline{T^{S^2}})\\
&\simeq\Hom_{H^*}(k,\underline{T^{S^2}}\ot\Hom^L(H,W))\\
&\simeq T^{S^2}\ot\Hom(H,W)^H
\end{align*} where $\Hom_{H^*}(k, T\ot V)\simeq \Hom_{H^*}(k, V\ot T^{S^2})$ is due to the rigidity of $\Mc^H_{fd}$ and the isomorphism $T^{**}\simeq T^{S^2}$.

Analogously we have:   \begin{align*}\Hom(H,W\ot L)^H&\simeq\Hom_{H^*}(k,\Hom^R(H,W\ot L))\\
&\simeq\Hom_{H^*}(k,\Hom^R(H,W)\ot\overline{L})\\
&\simeq\Hom_{H^*}(k,\overline{L^{S^{-2}}}\ot\Hom^R(H,W))\\
&\simeq\Hom_{H^*}(k,\Hom^R(H,W)\ot\underline{L^{S^{-2}}})\\
&\simeq \Hom(H,W)^H\ot L^{S^{-2}}.
\end{align*}

In the latter we have used the analogues of \eqref{iso1} and \eqref{iso2}; namely the isomorphisms: $$\Hom^R(H,W\ot L)\to\Hom^R(H,W)\ot\overline{L}$$ $$f\ot l\mapsto f\ot l$$ and $$\overline{L}\ot\Hom^R(H,W)\to \Hom^R(H,W)\ot \underline{L}$$ $$l\ot f\mapsto f(-l_1)\ot l_0.$$

The result now follows after tracing through the isomorphisms.
\end{proof}

Denote by $\Mc^{H\#}$ the $\Mc^H_{fd}$ bimodule category with $$T\cdot M\cdot L=T\ot M\ot L^{S^2}$$ and by ${^{\#}\contra}$ the $\Mc^H_{fd}$ bimodule category with $$T\cdot N\cdot L=T^{S^2}\ot N\ot L$$ then we immediately obtain the following as a Corollary of Proposition \ref{bimoduleprop}:

\begin{corollary}\label{adjunctionscor}
The functors $$\widehat{(-)}: \Mc^{H\#}\to {^{\#}\contra}$$  and $$(-)':{^{\#}\contra}\to\Mc^{H\#}$$ are bimodule functors over $\Mc^H_{fd}$ and so induce functors between the corresponding centers of bimodule categories.
\end{corollary}
\begin{proof}
The claim about $\widehat{(-)}$ is immediate from Proposition \ref{bimoduleprop}.  Since $\Mc^H_{fd}$ is rigid, the statement about $(-)'$ follows from the one about $\widehat{(-)}$ through adjunction juggling, since they are adjoint functors.

\end{proof}

\begin{remark}\label{explicit}
The adjunction manipulations mentioned in the proof of Corollary \ref{adjunctionscor} can be traced through to obtain an explicit analogue of Proposition \ref{bimoduleprop} for the functor $N\mapsto N'$.  Namely, for $N\in\contra$ and $T,L\in\Mc^H_{fd}$ we have a natural isomorphism in $\Mc^H$: $$H\odot_H(T\ot N\ot L)\simeq T^{S^{-2}}\ot(H\odot_H N)\ot L^{S^2}$$ $$h\ot t\ot n\ot l\mapsto t_0\ot S^{-1}(t_1)hS(l_1)\ot n\ot l_0.$$
\end{remark}

As in \cite{hks}, we have central interpretations of $aYD$ objects.

\begin{lemma}\label{comodulecenter}
The center of $\Mc^{H\#}$ is equivalent to the category of anti-Yetter-Drinfeld modules, namely $$\Zc_{\Mc^H_{fd}}(\Mc^{H\#})\simeq aYD\text{-mod}.$$
\end{lemma}

\begin{proof}
The proof proceeds very much like that of Proposition \ref{aydprop2} and so we provide only a sketch.  Let $M\in{_H\Mc^H}$, i.e., it is both a left module and a right comodule, and let $T\in\Mc^H_{fd}$.  Consider the map $$\tau:T\ot M\to M\ot T^{S^2}$$ $$t\ot m\mapsto t_1 m\ot t_0.$$  It is an isomorphism with inverse $m\ot t\mapsto t_0\ot S(t_1)m$.  It is a map in $\Mc^H$ if and only if $M\in aYD\text{-mod}$.  It is immediate that $(M,\tau)\in \Zc_{\Mc^H_{fd}}(\Mc^{H\#})$.

Conversely, let $M\in\Mc^H$ such that we have natural isomorphisms $\tau_T:T\ot M\to M\ot T^{S^2}$ for all $T\in \Mc^H_{fd}$.  Now proceed in a by now familiar fashion.  We need an action $H\ot M\to M$ which we obtain via a limit over finite dimensional subcoalgebras $C\subset H$, i.e., \begin{align*}\Hom(H\ot M,M)&=\Hom((\varinjlim C)\ot M,M)\\&=\Hom(\varinjlim (C\ot M),M)\\&=\varprojlim\Hom(C\ot M,M)\end{align*} and the latter contains $(Id\ot\epsilon_C)\circ\tau_C$.
\end{proof}

\begin{remark}
Note that what these limit arguments demonstrate is that in contrast to the $H$-module case considered in \cite{sk}, the $H$-comodule case is much easier as it reduces to the rigid category $\Mc^H_{fd}$.  More precisely, Proposition \ref{aydprop2} shows that $\Zc_{\Mc^H}(\Mc^{H op})\simeq aYD\text{-mod}$ by essentially showing that $\Zc_{\Mc^H}(\Mc^{H op})\simeq \Zc_{\Mc^H_{fd}}(\Mc^{H op})$, but the latter is clearly $\Zc_{\Mc^H_{fd}}(\Mc^{H\#})$, which as we saw above is equivalent to $\Zc_{\Mc^H}(\Mc^{H\#})$.
\end{remark}

\begin{lemma}\label{contramodulecenter}
The center of ${^{\#}\contra}$ is equivalent to the category of anti-Yetter-Drinfeld contramodules, namely $$\Zc_{\Mc^H_{fd}}({^{\#}\contra})\simeq aYD\text{-\contramod}.$$
\end{lemma}

\begin{proof}
Repeat the proof of Lemma \ref{comodulecenter} verbatim with the exception that $$\tau: T^{S^2}\ot N\to N\ot T$$ $$t\ot n\mapsto t_1 n\ot t_0$$ is a map in $\contra$ if and only if $N\in aYD\text{-\contramod}$.
\end{proof}

We summarize this section with the following Theorem.

\begin{theorem}\label{summarythm} The following diagram commutes:
$$\xymatrix{\Zc_{\Mc^H_{fd}}(\Mc^{H\#})\ar@<+.5ex>[rr]^{\Zc_{\Mc^H_{fd}}(\widehat{(-)})} & & \Zc_{\Mc^H_{fd}}({^{\#}\contra})\ar@<+.5ex>[ll]^{\Zc_{\Mc^H_{fd}}((-)')}\\
aYD\text{-mod}\ar[u]^{Lemma\,\, \ref{comodulecenter}}_\simeq\ar@<+.5ex>[rr]^{\widehat{(-)}_H} & & aYD\text{-\contramod}\ar@<+.5ex>[ll]^{(-)_H'}\ar[u]^\simeq_{Lemma\,\, \ref{contramodulecenter}}\\
}$$ Recall that for $M\in aYD\text{-mod}$ we equip $\widehat{M}$ with \eqref{contraaction} and \eqref{weird action}, whereas for $N\in aYD\text{-\contramod}$ we equip $N'$ with \eqref{coaction} and \eqref{notsoweird action}.
\end{theorem}
\begin{proof}
For the $\widehat{(-)}$ case we have the map $T\ot M\to M\ot T^{S^2}$ with $t\ot m\mapsto t_1m\ot t_0$ mapping to $\Hom(H,T\ot M)^H\to\Hom(H,M\ot T^{S^2})^H$ with $t\ot f\mapsto t_1f\ot t_0$ which maps under the identification of Proposition \ref{bimoduleprop} to $T^{S^2}\ot \Hom(H,M)^H\to \Hom(H,M)^H\ot T$ with $t_0\ot f(S^2(t_1)-)\mapsto t_2f(-t_1)\ot t_0$ and the latter coincides with $t\ot g\mapsto t_1\cdot g\ot t_0$.

For the adjoint $(-)'$ we have $T^{S^2}\ot N\to N\ot T$ with $t\ot n\mapsto t_1n\ot t_0$ mapping to $H\odot_H(T^{S^2}\ot N)\to H\odot_H(N\ot T)$ with $h\ot t\ot n\mapsto h\ot t_1n\ot t_0$ which identifies with $T\ot(H\odot_H N)\to (H\odot_H N)\ot T^{S^2}$ with $t_0\ot S(t_1)h\ot n\mapsto hS(t_1)\ot t_2 n\ot t_0$ under the isomorphisms of Remark \ref{explicit} and the latter coincides with $t\ot (x\ot m)\mapsto t_1\cdot(x\ot m)\ot t_0$.
\end{proof}

In the end we see that $((-)_H',\widehat{(-)}_H)$ is an adjoint pair between $aYD\text{-mod}$ and $aYD\text{-\contramod}$ extending the result of \cite{pos}.

\subsubsection{Stability.}\label{stabilitysection}
Recall that in order to obtain cyclic cohomology we need to consider the coefficients in \emph{stable} anti-Yetter-Drinfeld modules or contramodules.  We now address the preservation of the stability conditions under the adjoint pair of functors of the previous section.

Recall the map $$\sigma_M: M\to M$$ $$m\mapsto m_1m_0$$ with the inverse $m\mapsto S^{-1}(m_1)m_0$; it defines an element $\sigma\in\Aut(Id_{aYD\text{-mod}})$.  Similarly, there is a $$\sigma_N: N\to N$$ $$n\mapsto\alpha(r_n)$$ with the inverse $n\mapsto\alpha(h\mapsto S^{-1}(h)n)$; it defines an element $\sigma\in\Aut(Id_{aYD\text{-\contramod}})$.

It is an easy calculation to see that $\widehat{\sigma_M}:\widehat{M}\to\widehat{M}$ coincides with $\sigma_{\widehat{M}}:\widehat{M}\to\widehat{M}$ and also $(\sigma_N)'=\sigma_{N'}$.  For example to prove the latter equality observe that the left hand side is $h\ot n\mapsto h\ot\alpha(r_n)=h^2\ot r_n(h^1)=h^2\ot h^1 n=(h\ot n)_1(h\ot n)_0$ which is the right hand side.

Recall that $saYD\text{-mod}$ is the full subcategory of $aYD\text{-mod}$ that consists of $M$ such that $\sigma_M=Id_M$.  The definition of $saYD\text{-\contramod}$ is identical.  We have proved the following Corollary to Theorem \ref{summarythm}:

\begin{corollary}
The functors $((-)_H',\widehat{(-)}_H)$ is an adjoint pair between $saYD\text{-mod}$ and $saYD\text{-\contramod}$.
\end{corollary}

\section{A comodule-contramodule correspondence.}\label{correspondence}
Here we will address the question of $\widehat{(-)}$ (equivalently $(-)'$) being an equivalence.  Note that in light of the preceding discussion if $\widehat{(-)}:\Mc^H\to\contra$ is an equivalence, then so is $\widehat{(-)}_H: aYD\text{-mod}\to aYD\text{-\contramod}$ and also $\widehat{(-)}_H: saYD\text{-mod}\to saYD\text{-\contramod}$.

As usual, let us consider $k$ as the trivial $H$-comodule, and let $J=\widehat{k}$ be its contramodule image under $\widehat{(-)}$.  Note that this is nothing but the two-sided ideal in $H^*$ consisting of right integrals \cite{integral}.  Namely, $\chi\in J$ if and only if we have $\chi(h^1)h^2=\chi(h)1$ for all $h\in H$.  Strictly speaking it is left integrals that are considered in \cite{integral} but if $\chi$ is a left integral then $\chi(S(-))$ is right and vice versa.  It is known \cite{hopf} that $\text{dim} J\leq 1$ and if $J\neq 0$ then $S$ is invertible, which we have been assuming anyhow.

\begin{remark}\label{bigk}
Dually, we may consider $k$ as the trivial contramodule, i.e., $\alpha: H^*\to k$ is evaluation at $1\in H$.  Let $K=k'$ and note that $K=H/I$ where $I$ is generated by $\mu(h^1)h^2-\mu(1)h$ for $\mu\in H^*$ and $h\in H$. Thus $K^*=I^\perp=\{\chi\in H^*|\mu(1)\chi(h)=\mu(h^1)\chi(h^2) \forall h\}$ and the latter is the ideal of left integrals.
\end{remark}

We are ready for the first negative result:
\begin{lemma}\label{neg1}
If $J=0$ then $\widehat{(-)}$ is not an equivalence.
\end{lemma}
\begin{proof}
Obviously we have that $\widehat{k}=0$, but furthermore, by Proposition \ref{bimoduleprop} we have that for $M\in\Mc^H_{fd}$, $\widehat{M}\simeq M^{S^2}\ot J=0$.
\end{proof}

On the other hand let us assume that $J\neq 0$.    Let $\contra_{rfd}$ denote the full subcategory of $\contra$ consisting of finite dimensional, rational contramodules.  By analogy with the $H^*$-module case, we say that a finite dimensional contramodule $M$ is \emph{rational} if the structure map $\alpha$ factors through $\Hom(C,M)$ for some $C$ a finite dimensional subcoalgebra of $H$.

\begin{lemma}\label{neg2}
Let $J\neq 0$ then $\widehat{(-)}:\Mc_{fd}^H\simeq\contra_{rfd}$.
\end{lemma}
\begin{proof}
Again, for $M\in\Mc_{fd}^H$ we have that $\widehat{M}\simeq M^{S^2}\ot J$.  Note that by \cite{integral} the contramodule $J$ is rational and thus so is $\widehat{M}$.  On the other hand any rational finite dimensional contramodule is essentially a comodule (see Lemma \ref{rationalcontra}) and so $(-\ot{^*J})^{S^{-2}}$ is the inverse of $\widehat{(-)}$.
\end{proof}

The above Lemma should be considered as in general a negative result.  Namely, if exotic, i.e., non-rational contramodules are possible, then the equivalence fails.  More precisely, let us consider the possibility of exotic contramodule structures on $k$.  Let $\chi\in J$ and observe that $$\alpha(x\mapsto\alpha(y\mapsto\chi(xS(y))))=\alpha(h\mapsto\chi(h^1S(h^2)))=\alpha(h\mapsto\epsilon(h))\chi(1)=\chi(1).$$  Since by \cite{integral}, as $x$ ranges over $H$, the functional $\chi(xS(-))$ ranges over $H^{*rat}$ so if $$\chi(1)\neq 0$$ then $\exists\mu\in H^{*rat}$ such that $\mu\cdot 1=c\neq 0$.  So that for any $\eta\in H^*$ we have $\eta\cdot 1=\eta\mu\frac{1}{c}=\eta(\mu_1)\mu_0\frac{1}{c}$ and so the action of $H^*$ on $k$ factors through $C^*$ and the structure on $k$ is necessarily rational.  On the other hand if  $\chi(1)=0$ then it is possible that the whole of $H^{*rat}$ acts trivially without $H^*$ doing the same, resulting in an exotic structure.

This suggests two possibilities for $\widehat{(-)}$ being an equivalence: \begin{itemize}\medskip
\item $\exists\chi\in J$ with $\chi(1)\neq 0$.\smallskip
\item $H$ is finite dimensional.
\end{itemize}

Note that the second case may appear trivial at first, but it isn't.  It is true that there is no difference between $H$-comodules, $H^*$-modules and $H$-contramodules in the case when $H$ is finite dimensional.  However, we are not interested in the naive identification of the categories, rather the  $\widehat{(-)}$ one.  The latter functor is the one that translates the equivalence between comodules and contramodules to the equivalence between the $saYD$ versions that we need.  Of course given all the work already done on this matter, the conclusion is easy to obtain, so we start with this case.

\begin{proposition}\label{finitecase}
Let $H$ be finite dimensional, then $\widehat{(-)}$ is an equivalence, and so is $\widehat{(-)}_H$.
\end{proposition}
\begin{proof}
From \cite{integral} we know that $J\neq 0$.  Furthermore, for $M\in\Mc^H$ we have $M=\varinjlim M_i$ with $M_i\in\Mc^H_{fd}$ so that $\widehat{M}=\Hom(H,M)^H=\Hom(H,\varinjlim M_i)^H$ which by the finite dimensionality of $H$ is $\varinjlim\Hom(H,M_i)^H\simeq\varinjlim(M_i^{S^2}\ot J)=M^{S^2}\ot J$.  Since there are no exotic contramodules here this proves the equivalence.
\end{proof}

Moving on to the first case we get by \cite{integral} that the $\chi(1)\neq 0$ condition is actually very strict.  Namely, we have that $H$ is such that as a coalgebra $H=\bigoplus_i C_i$ where $C_i$ are finite dimensional simple subcoalgebras. Let $\epsilon_i$ denote the counit of $C_i$ with $\epsilon=\sum\epsilon_i$.  For $x\in H$ let $x=\sum_i x_i$ denote its decomposition with respect to that of $H$.

\begin{theorem}\label{semisimple}
The category of $H$-comodules and $H$-contramodules are equivalent.  The former consists of $\bigoplus_i M_i$ and the latter of $\prod_i M_i$ where $M_i$ are right $C_i$-comodules, i.e.,  $M_i\in\Mc^{C_i}$.
\end{theorem}
\begin{proof}
The assertion about the comodules is immediate.  Now let $M$ be an $H$-contramodule, define $\alpha_i:M\to M$ via $\alpha_i(m)=\alpha(\epsilon_i(-)m)$.  Note that $$\alpha_i(\alpha_j(m))=\alpha(x\mapsto\epsilon_i(x)\alpha(y\mapsto\epsilon_j(y)m))=\alpha(h\mapsto\epsilon_i(h^1)\epsilon_j(h^2)m)=\delta_{ij}\alpha_i(m).$$
Let $M_i=\alpha_i(M)$ and consider $\beta: M\to\prod M_i$ such that $$\beta(m)=(\alpha_i(m))_i$$ and $\iota:\prod M_i\to\Hom(H,M)$ via $$\iota((m_i)_i)(x)=\sum\epsilon_i(x)m_i.$$  We have that \begin{align*}\alpha\iota\beta(m)&=\alpha(x\mapsto\sum\epsilon_i(x)\alpha_i(m))\\
&=\alpha(x\mapsto\sum\epsilon_i(x)\alpha(y\mapsto\epsilon_i(y)m))\\&=\alpha(h\mapsto\sum\epsilon_i(h^1)\epsilon_i(h^2)m)\\&=\alpha(h\mapsto\epsilon(h)m)=m.
\end{align*}
On the other hand we have that $\beta\alpha\iota((m_i)_i)=(\alpha_i(\alpha(x\mapsto\sum\epsilon_j(x)m_j)))_i$ and so we need to show that $$m_i=\alpha_i(\alpha(x\mapsto\sum\epsilon_j(x)m_j)),$$ but the latter is \begin{align*}\alpha(y\mapsto\epsilon_i(y)\alpha(x\mapsto\sum\epsilon_j(x)m_j))&=\alpha(h\mapsto\sum_j\epsilon_i(h^1)\epsilon_j(h^2)m_j)\\
&=\alpha(h\mapsto\epsilon_i(h^1)\epsilon_i(h^2)m_i)\\
&=\alpha(\epsilon_i(-)m_i)=\alpha_i(m_i)=m_i.
\end{align*}
Thus $\beta:M\simeq\prod_i M_i$ and using this identification we see that $\alpha:\Hom(H,M)\to M$ becomes $$\prod_i\Hom(H,M_i)\to\prod_i M_i$$ $$(f_i)_i\mapsto (\alpha(h\mapsto f_i(h_i)))_i$$ so that if we denote by $\alpha^i:\Hom(C_i,M_i)\to M_i$ the map $\alpha^i(f)=\alpha(h\mapsto f(h_i))$ then we see that the original $\alpha$ identifies with $\prod_i\alpha^i:\prod_i\Hom(C_i,M_i)\to\prod_i M_i.$  It is immediate that $\alpha^i$ is a $C_i$-contramodule structure on $M_i$ and since $C_i$ is finite dimensional is the same as a $C_i$-comodule structure.

Conversely, given the data of $\rho_i: M_i\to M_i\ot C_i$ we can define $$\alpha^i:\Hom(C_i,M_i)=M_i\ot C_i^*\to M_i$$  $$m\ot\chi\mapsto \chi(m_1)m_0$$ and assemble the $\alpha^i$ into an $\alpha:\Hom(H,\prod_i M_i)\to\prod_i M_i$ that satisfies the contramodule axioms.

Now let $\phi:M\to N$ be a map of contramodules and let $m\in M$ with $m=(m_i)_i$ under the $\beta$ identification, then $$\phi(m)_i=\alpha_N(\epsilon_i(-)\phi(m))=\phi(\alpha_M(\epsilon_i(-)m))=\phi(m_i)$$ so that $\phi=\prod_i\phi_i$ with $\phi_i: M_i\to N_i$. It is immediate that $\phi_i\in\Hom(M_i,N_i)^{C_i}$ and that conversely, any such $(\phi_i)_i$ data can be reassembled into a $\phi:M\to N$ a map of contramodules.

\end{proof}

\begin{remark}\label{contradiff}
The proof of Theorem \ref{semisimple} demonstrates a difference between $H$-contramodules and $H^*$-modules.  While there is a forgetful functor from the former to the latter, the contramodule condition is better behaved than the $H^*$-module one in the case of the infinite dimensional $H$.   Considering finite dimensional contramodules, that at first glance appear to be given an action indistinguishable from that of an $H^*$-module, it is the associativity that is strictly strengthened in the contramodule case.  More precisely, there exist exotic $1$-dimensional $kG^*$-modules (for example given by non-principal ultrafilters on $G$), yet any $1$-dimensional $kG$-contramodule is supported at some $g\in G$, just as is the case for $kG$-comodules.  The difference is due to the fact that in the contramodule case we have the freedom to work with the full $(H\ot H)^*$ as opposed to only $H^*\ot H^*$.   Of course in the case when $H$ is finite dimensional all three categories: $H$-contramodules, $H$-comodules and $H^*$-modules are equal.
\end{remark}

We need to connect the above to our adjoint pair of functors.

\begin{proposition}\label{semisimplecase}
The correspondence $$\bigoplus_i M_i\leftrightarrow\prod_i M_i$$ of Theorem \ref{semisimple} is given, up to equivalence, by the adjoint functor pair $(\widehat{(-)},(-)')$.  Thus $\widehat{(-)}$ is an equivalence and so is $\widehat{(-)}_H$.
\end{proposition}

\begin{proof}
Observe that $$\widehat{M}=\Hom(H,M)^H=\prod_i\Hom(C_i,M_i)^{H_i}=\prod_i\Hom_{A_i}(A_i^*,M_i)$$ where $A_i=C^*_i$ is a unital simple finite dimensional algebra.  Let $$\mu_i:A_i\to A_i^*$$ be given by $\mu_i(a)(b)=\text{tr}_{A_i}(l_{ab})$, i.e., it is the trace of left multiplication by $ab\in A_i$.  Note that $\mu_i$ is an $A_i$-bimodule map and $\mu_i(1)(1)=\text{tr}_{A_i}(1)=\text{dim} A_i\neq 0$ since $\text{char}k=0$, so that $\mu_i$ is an isomorphism by the simplicity of $A_i$.  So   $\widehat{M}\simeq\prod_i\Hom_{A_i}(A_i,M_i)\simeq\prod_i M_i$.

Similarly $$N'=\bigoplus_i C_i\ot_{C_i^*} N_i=\bigoplus_i A_i^*\ot_{A_i}N_i\simeq\bigoplus_i N_i.$$
\end{proof}

\subsection{The case of $H=kG$.}

Let $G$ be an infinite discrete group.  We ask that $G$ be infinite as otherwise all of our considerations here become more or less trivial.  Let $M$ be a $kG$-contramodule, i.e., we view $kG$ as a counital coalgebra with $\Delta(g)=g\ot g$ and $\epsilon(g)=1$. We have the following corollary of Theorem \ref{semisimple}:

\begin{corollary}
The category of $kG$-contramodules $\contrakg$ is equivalent to the category of $G$-graded vector spaces $\vect_G$. The equivalence is given by $$\Gamma(G,-):\vect_G\to \contrakg.$$ Compare this with the well known equivalence $$\Gamma_c(G,-):\vect_G\to\Mc^{kG}$$ where $\Gamma_c$ are global sections with compact support.
\end{corollary}
\begin{proof}
Note that $kG=\bigoplus_{g\in G}kg$ with $kg=k$ as coalgebras.
\end{proof}

It is well known that the category of anti Yetter-Drinfeld modules for $kG$ (since $S^2=Id$ it coincides with the category of Yetter-Drinfeld modules, and thus with the center of the monoidal category of $kG$-modules) is equivalent to the category $\vect_{G/G}$ of $G$-equivariant $G$-graded vector spaces.  More precisely, the $kG$-comodule part of the structure gives the $G$-grading, and the $kG$-module part gives the $G$-action, while the Yetter-Drinfeld compatibility ensures that the action obeys $$x:M_g\to M_{xgx^{-1}}.$$  We have an immediate Corollary to Proposition \ref{semisimplecase}:

\begin{corollary}\label{thm1}
The category of $aYD$-contramodules for $kG$ is equivalent to the category of $G$-graded $G$-equivariant vector spaces via $$\Gamma(G,-):\vect_{G/G}\to aYD\text{-\contramod}.$$
\end{corollary}

We now would like to address the question of stability.  A stable $aYD$-module for $kG$ is known to be $G$-graded $G$-equivariant vector space with the stability condition translating into \begin{equation}\label{stability}x\cdot m_x=m_x\end{equation} for all $x\in G$ and all $m_x\in M_x$.  Denote by $\vect'_{G/G}$ the full subcategory of $\vect_{G/G}$ consisting of objects for which \eqref{stability} holds. We have another immediate Corollary to Proposition \ref{semisimplecase}:

\begin{corollary}\label{kgeq}
The functor $$\Gamma(G,-):\vect'_{G/G}\to saYD\text{-\contramod}$$ is an equivalence.
\end{corollary}

We can now restate the Proposition \ref{dualityprop1} more elegantly in the case of $H=kG$.

\begin{proposition}\label{dualityprop}
Let $A$ be a $G$-equivariant algebra, and $\Mc\in \vect'_{G/G}$.  Then $$\widehat{HC}^n_G(A, \Gamma(G,\Mc))\simeq HC^{n,G}(A\rtimes G, \Gamma_c(G,\Mc)).$$
\end{proposition}

\begin{remark}
While the right hand side of the above Proposition is definition invariant, the left hand side uses the definition of \cite{s} and not the more classical one used in \cite{hks}.
\end{remark}

\subsection{A new ``coefficient".}\label{nonrep}
Since the introduction of coefficients in symmetric $2$-contratraces in \cite{hks}, there remained an obvious question: do these simply generalize the already well known coefficients in $saYD$-modules or contramodules to other settings, or do these traces furnish examples of coefficients that had not yet been considered even in the classical theories?  In \cite{s} we gave a derived version of the definition of cyclic cohomology with coefficients that restricted the possible symmetric $2$-contratraces to the left exact ones.  The results obtained in \cite{sk} immediately tell us that in the case of $H$-module algebras we need to look beyond the representable   symmetric $2$-contratraces if we are to obtain anything but the usual $saYD$-contramodule coefficients.  In the present paper, Corollary \ref{reptrace1} implies the same about $H$-comodule algebras, i.e., we need a non-representable contratrace to get away from the usual $saYD$-module coefficients.  We will construct one below.

Let $G$ have infinitely many conjugacy classes (such as when $G=\mathbb{Z}$ for example).  Let $\Mc_{\left<g\right>}\in \vect'_{G/G}$ be supported on the conjugacy class $\left<g\right>$, for example we can let $$(\Mc_{\left<g\right>})_x=\begin{cases}k,\quad x\in\left<g\right>\\0,\quad else\end{cases}$$ with the trivial $G$-action.  Then each $\Mc_{\left<g\right>}$ yields a representable left exact symmetric $2$-contratrace $$\Fc_{\left<g\right>}(V)=\Hom(V,\Gamma_c(G,\Mc_{\left<g\right>}))^G,$$ yet $$\bigoplus_{\left<g\right>}\Fc_{\left<g\right>}:V\mapsto \bigoplus_{\left<g\right>}\Fc_{\left<g\right>}(V)$$ is an example of a non-representable, left exact symmetric $2$-contratrace on $\Mc^H$.  Note that taking $\Mc$ to be the superposition of all $\Mc_{\left<g\right>}$'s would result in $V\mapsto \prod_{\left<g\right>}\Fc_{\left<g\right>}(V)$.

\section{Periodicities.}\label{period}

In this section we revisit the $YD_i$-modules from \cite{hks} and see that under the conditions that we have been looking at in this paper, there is nothing new that arises and we still only have the Yetter-Drinfeld and the anti-Yetter-Drinfeld modules and contramodules; this is the first observed periodicity.  In addition, we examine a natural symmetry on these objects and observe that it too is periodic; this we refer to as the second periodicity.

We recall  the definition of $YD_i$-modules:

\begin{definition}
Let $M$ be a left module and a right comodule over $H$, and let $i\in\mathbb{Z}$.  We say that $M$ is a $YD_i$-module if
\begin{equation}\label{genyd2}
  (hm)_{0}\ot(hm)_{1}=h^{2}m_{0}\ot h^{3}m_{1} S^{-1-2i}(h^{1}),
  \end{equation} for $h\in H$ and $m\in M$.
\end{definition}
\begin{remark}
Equivalently, we can define $YD_i$-modules by requiring that the comodule structure map $M\to M\ot H$ is $H$-equivariant with respect to the $H$-structure on the right hand side given by $x\cdot(m\ot h)=x^2m\ot x^3 h S^{-1-2i}(x^{1})$.
\end{remark}

Note that $YD_{-1}$-modules are anti-Yetter-Drinfeld modules, while $YD_0$-modules are Yetter-Drinfeld modules.

We can rephrase the above a little more conceptually.  Let $\mathbb{Z}$ act on $\Mc^H$ with $1\cdot M=M^{S^2}$ so that we may consider the monoidal category $\Mc^H\rtimes\mathbb{Z}$.  We get an immediate generalization of Lemma \ref{comodulecenter}:

\begin{lemma}
We have an equivalence of monoidal categories: $$\Zc_{\Mc^H_{fd}}(\Mc^H\rtimes\mathbb{Z})\simeq\bigoplus_{i\in\mathbb{Z}}YD_{-i}\text{-mod}.$$
\end{lemma}

There are a few consequences of the above.  First, if $M\in YD_{i}\text{-mod}$ and $N\in YD_{j}\text{-mod}$ then $M\ot N\in YD_{i+j}\text{-mod}$ with the usual comodule structure, but ${_{S^{-2i}}N}\ot M$ as an $H$-module, i.e., $$(m\ot n)_0\ot(m\ot n)_1=m_0\ot n_0\ot m_1n_1$$ but $$x\cdot(m\ot n)=x^2m\ot S^{-2i}(x^1)n.$$  Second, if $M\in YD_{i}\text{-mod}$ then so is $1\cdot M={_{S^{-2}}M}^{S^2}\in YD_{i}\text{-mod}$.  Third,  $\Mc^H$ has internal Homs, and so does $\Mc^H\rtimes\mathbb{Z}$, i.e., $$\Hom^l((M,j),(N,i))=(\Hom^l((i-j)M,N),i-j)$$ and the same for right Homs.  Consequently, $\Zc_{\Mc^H_{fd}}(\Mc^H\rtimes\mathbb{Z})$ has internal Homs as well.  In particular $\Mc^H_{fd}$ is rigid, so is $\Mc_{fd}^H\rtimes\mathbb{Z}$ with $(V,i)^*=((-i)V^*, -i)$ and ${^*(V,i)}=((-i){^*V}, -i)$ and so is $\Zc_{\Mc^H_{fd}}(\Mc_{fd}^H\rtimes\mathbb{Z})$.  Thus  if $M\in YD^{fd}_{i}\text{-mod}$ then we have elements $M^\star$ and ${^\star M}$ in $YD_{-i}^{fd}\text{-mod}$ that are its right and left duals.

Just as we have generalized $aYD\text{-mod}$ to $YD_i\text{-mod}$, we can do the same to $aYD\text{-\contramod}$.

\begin{definition}
Let $M$ be a left $H$-module and a right $H$-contramodule, we say that $M$ is a $YD_i$-contramodule if the contramodule structure $\alpha:\Hom(H,M)\to M$ is $H$-equivariant with respect to the $H$-action on the left given by
\begin{equation}\label{genyd4}
  h\cdot f=h^2f(S(h^3)-S^{2-2i}(h^1)),
  \end{equation} where $h\in H$ and $f\in \Hom(H,M)$.
\end{definition}

Note that $YD_1$-contramodules are $aYD$-contramodules.

\subsection{The first periodicity.}

We can easily generalize the content of Section \ref{adjunctionssection} as follows.  We have the Proposition/Definition below.

\begin{proposition}\label{functor}
Let $M$ be a $YD_{i-1}$-module, define $\widehat{M}=\Hom(H,M)^H$ and equip the latter with a left $H$-action via $$h\cdot f=h^2f(S(h^3)-S^{-2i}(h^1))$$ and an $H$-contraaction as before \eqref{contraaction}.  Let $N$ be a $YD_{i+1}$-contramodule, define $N'=H\odot_H N$ and equip the latter with a left $H$-action via $$h\cdot (x\ot n)=S^{1-2i}(h^3)xS(h^1)\ot h^2 n$$ and an $H$-coaction as before \eqref{coaction}.

This defines an adjoint pair of functors $((-)_H',\widehat{(-)}_H)$: $$\xymatrix{YD_{i-1}\text{-mod}\ar@<+.5ex>[rr]^{\widehat{(-)}_H} & &  YD_{i+1}\text{-\contramod}\ar@<+.5ex>[ll]^{(-)_H'}.}$$
\end{proposition}

Let us again (see Remark \ref{bigk} and the preceding discussion) consider the trivial $YD_0$-module $k$ from which we obtain by the Proposition \ref{functor} the object $J=\widehat{k}$ which is now seen to be in $YD^{fd}_2\text{-\contramod}$. Conversely, again considering the trivial $YD_0$-contramodule $k$, we obtain $K=k'$ which is now seen to be in $YD^{fd}_{-2}\text{-mod}$.  If we denote by $YD_i^{rfd}\text{-\contramod}$ the full subcategory of $YD_i\text{-\contramod}$ that consists of objects that as contramodules are in $\contra_{rfd}$ then $J\in YD_2^{rfd}\text{-\contramod}$.  Observe that we have an easy Lemma:

\begin{lemma}\label{rationalcontra}
We have an equivalence (equality actually) of categories: $$\iota: YD_i^{fd}\text{-mod}\to YD_i^{rfd}\text{-\contramod}$$ that does not change the underlying vector space $M$, nor the $H$-action, and sends the coaction to the contraaction: $$\Hom(H,M)=H^*\ot M\to M$$ $$\chi\ot m\mapsto \chi(S^2(m_1))m_0.$$
\end{lemma}

As a consequence, we have $\iota^{-1}J\in YD_2^{fd}\text{-mod}$ which is the dual of $K\in YD_{-2}^{fd}\text{-mod}$.

\begin{remark}\label{periodicity 1}
If $H$ is a Hopf algebra with $J\neq 0$ then both $YD_i\text{-mod}$ and $YD_i\text{-\contramod}$ are $2$-periodic, i.e., $$J\ot-:YD_i\text{-mod}\simeq YD_{i+2}\text{-mod}$$ and the same for contramodules.
\end{remark}

For a finite dimensional $H$, the functor $\widehat{(-)}_H$ is essentially the periodicity above.  Not so for the infinite dimensional case.

\subsection{The second periodicity.}
Recall our discussion of stability in Section \ref{stabilitysection}.  We observe that the $$\sigma\in\Aut(Id_{aYD\text{-mod}})$$
that was used to define stability for $aYD$-modules (and its contramodule variant) can be generalized, with an interesting difference, to an arbitrary $i$ for both modules and contramodules.  More precisely,  $$\sigma\in\text{Iso}(Id_{YD_{i-1}\text{-mod}},(-i)\cdot),$$ i.e., for $M\in YD_{i-1}\text{-mod}$ we have  $\sigma_M: M\to{_{S^{2i}}M^{S^{-2i}}}$, with $m\mapsto S^{2i}(m_1)m_0$ and the inverse $m\mapsto S^{-1}(m_1)m_0$, is an identification in $YD_{i-1}\text{-mod}$.

\begin{remark}\label{second periodicity}
The above implies that if $M$ is a Yetter-Drinfeld module, then it is canonically isomorphic to ${_{S^{-2}}M^{S^2}}$ as a Yetter-Drinfeld module.  More generally, the action of $\mathbb{Z}$ on $YD_{i-1}\text{-mod}$ factors through $\mathbb{Z}/i\mathbb{Z}$.  We will see below that the same holds for $YD_{i+1}\text{-\contramod}$.
\end{remark}

Note that the $M\mapsto{_{S^{-2}}M^{S^2}}$ symmetry of $YD_i$-modules also exists for $YD_i$-contramodules, i.e.,  ${_{S^{-2}}N^{S^2}}$ has its $H$-action modified by $S^{-2}$ and $\alpha^{S^2}(f)=\alpha(f(S^2(-)))$.  Then we have $$\sigma\in\text{Iso}(Id_{YD_{i+1}\text{-\contramod}},(-i)\cdot),$$ i.e., for $N\in YD_{i+1}\text{-\contramod}$ we have  $\sigma_N: N\to{_{S^{2i}}N^{S^{-2i}}}$, with $n\mapsto \alpha(r_n)$ and the inverse $n\mapsto \alpha(h\mapsto S^{2i-1}(h)n)$, is an identification in $YD_{i+1}\text{-\contramod}$.

Furthermore, the generalized functor $\widehat{(-)}_H$ of Proposition \ref{functor} is compatible with these symmetries, namely the diagram of isomorphisms $$\xymatrix{\widehat{M}\ar[rr]^{\widehat{\sigma_M}}\ar[dr]_{\sigma_{\widehat{M}}} & & \widehat{{_{S^{2i}}M^{S^{-2i}}}}\\
& {_{S^{2i}}\widehat{M}^{S^{-2i}}}\ar[ur]_{\,\,f\mapsto f(S^{2i}(-))} &
}$$ commutes in $YD_{i+1}\text{-\contramod}$, where $M\in YD_{i-1}\text{-mod}$.


\bigskip

\noindent Department of Mathematics and Statistics,
University of Windsor, 401 Sunset Avenue, Windsor, Ontario N9B 3P4, Canada

\noindent\emph{E-mail address}:
\textbf{ishapiro@uwindsor.ca}

\end{document}